\UseRawInputEncoding
\documentclass{amsart}
\usepackage{amsthm}
\usepackage{amssymb}
\usepackage{amsfonts}
\usepackage{amsmath}

\usepackage{pinlabel}
\usepackage{graphicx}
\usepackage{color}

\newtheorem{lem}{Lemma}[section]

\newtheorem{thm}{Theorem}[section]
\newtheorem{coro}{Corollary}[section]
\newtheorem*{ques}{Question}
\newtheorem*{remark}{Remark}
\newtheorem*{acknow}{Acknowledgments}

\author{Feihuang Xia}
\title{simple closed geodesics in cusped hyperbolic {$3$}-manifolds}

\begin{document}
\begin{abstract}
We show that cusped finite-volume hyperbolic {$3$}-manifolds contain infinitely many simple closed geodesics.
\end{abstract}

\maketitle
\section{Introduction}\label{sec1}

A closed geodesic in a Riemannian manifold is said to be simple if it has no self-intersections, or nonsimple otherwise. It is an open question which Riemannian manifolds contain simple closed geodesics or which do not. In this paper, we focus on simple closed geodesics in hyperbolic manifolds, that is, complete Riemannian manifolds of constant curvature $-1$.

In any closed hyperbolic $3$-manifold, any homotopy class of a loop admits a unique geodesic representative, and any shortest closed geodesic must be simple.
On the other hand, some infinite-volume hyperbolic $3$-manifolds have no simple closed geodesics, such as $\mathbb{H}^3$. 
C. Adams, J. Hass and P. Scott showed that the Fuchsian group corresponding to the thrice-punctured sphere generates the only exception of a complete non-elementary orientable hyperbolic $3$-manifold that does not contain a simple closed geodesic \cite{MR1650997}.

Heuristically we expect many hyperbolic $3$-manifolds contain infinitely many simple closed geodesics. 
For example, S. M. Miller proved that many hyperbolic $3$-manifolds contain infinitely many simple closed geodesics. Examples include the figure-eight knot complement \cite{MR1917428}.
S. M. Kuhlmann showed that if a closed orientable hyperbolic $3$-manifold satisfies certain geometric and arithmetic conditions, it contains infinitely many simple closed geodesics \cite{MR2369199}. 
She actually found many such examples by program searching. 
In the cusped manifold case, S. M. Kuhlmann gave a general result. She showed that there are infinitely many simple closed geodesics in cusped finite-volume orientable hyperbolic $3$-manifolds \cite{MR2263061}. 

T. Chinburg and A. W. Reid proved that there exist infinitely many mutually incommensurable closed hyperbolic $3$-manifolds in which all primitive closed geodesies are simple by some arithmetic hyperbolic $3$-manifold theory \cite{MR1243786}, where we say two manifolds $M$ and $N$ are commensurable if there is a manifold that covers both $M$ and $N$ with ﬁnite degrees and incommensurable otherwise. 
They showed that if there is a nonsimple geodesic in a closed hyperbolic $3$-manifold $M=\mathbb{H}^3/\Gamma$, an invariant of the commensurability class of $\Gamma$ called invariant quaternion algebra must satify certain condition. They then construct infinitely many nonisomorphic quaternion algebras that do not meet this condition. However, it is still unknown whether all closed hyperbolic $3$-manifolds contain infinitely many simple closed geodesics.

In this paper, we extend Kuhlmann's result to higher dimensions with toral ends, and to the non-orientable, $3$-dimensional cusped case:

\begin{thm}\label{ThmZnCusp0}
    A cusped nonelementary hyperbolic $n$-manifold contains infinitely many simple closed geodesics if there is a cusp which is a flat $(n-1)$-manifold $\mathbb{R}^{n-1}/\Gamma_c$ such that $\Gamma_c \cong \mathbb{Z}^{n-1}$ for $n\ge 3$.
\end{thm}

\begin{thm}
    Any nonelementary hyperbolic $3$-manifold with a virtually rank-$2$ cusp contains infinitely many simple closed geodesics.
    In particular, any cusped finite-volume hyperbolic $3$-manifold contains infinitely many simple closed geodesics.
\end{thm}

Since it seems true that the self-intersections of a loop must be rare in high dimensional manifolds, there is a well-known question:
\begin{ques}
    Does every cusped hyperbolic $n$-manifold M contain infinitely many simple closed geodesics?
\end{ques}

Our method requires to analyse the cusp structure for constructing simple closed geodesics.
However, the cusp structure will be complicated as the dimension increases.
Our approach to construct simple closed geodesics depends on the location of self-tangencies.
So we will discuss different situations as the cusp structure or the self-tangencies change respectively.
These difficulties do not quite show up in the orientable $3$-dimensional case, but they will always appear when we consider higher dimensional cases.
Therefore, our case-by-case approach to the $3$-dimensional case does not automatically generalize to higher dimensions.

To prove these theorems, we follow the basic idea of Kuhlmann, and analyze structure of the cusp.
The cusp of a hyperbolic $n$-manifold $M$ admits a flat metric. Therefore, it can be represented as $\mathbb{R}^{n-1}/\Gamma_c$, where $\mathbb{R}^{n-1}$ contains Euclidean structure and $\Gamma_c$ is an isometric discrete group acting freely. 

When $M$ is a finite-volume hyperbolic orientable $3$-manifolds, $\Gamma_c \cong\mathbb{Z}^2$, generated by two parabolic transformations.
In \cite{MR2263061}, S.M. Kuhlmann considered an infinite family of closed geodesics that surround the cusp and showed that infinitely many of those are embedded.

To find this family, we consider the axis of isometries $t\circ g$, where $g$ is some certain isometry and $t$ is some isometry stabilizing a distinguished horocusp, where the axis is defined to be the geodesic which is fixed under a hyperbolic isometry.
We actually find out a specific family of such geodesics by letting $t$ vary along some carefully chosen direction.

To prove these geodesics are simple, we will estimate the locations of these axis first. Then we claim that the ends of these geodesics are limited in some certain small open balls when $t$ is `large' enough. Then we can study these geodesics more conveniently. That are Lemma \ref{lem1} and Lemma \ref{lem2}, which we will prove in the next section.

In this paper, we will use the similar method in \cite{MR2263061} to find this family of closed geodesics and make some extensions.
Then we deal with the $3$-dimensional non-orientable case.

This paper is organized as follows. In section \ref{sec2}, we define some notations and prove two lemmas which are mentioned above. In section \ref{sec3}, we extend the conclusion of \cite{MR2263061} to $n$-dimension. In section \ref{sec4}, we discuss the case of nonelementary hyperbolic $3$-manifolds with a virtually rank-$2$ cusp and find out infinitely simple closed geodesics. In section \ref{sec5}, we discuss some further work.

\begin{acknow}
    I would like to thank my supervisor Yi Liu for introducing me to this project, guiding me through my research and providing me many helpful comments and remarks along the way.  
\end{acknow}

\section{Preparation}\label{sec2}

In this section, we will introduce some preliminary notations and some elementary estimates.
The following notations will be used throughout this paper:
the horocusp $\mathcal{H}_0$ and $\mathcal{H}_\infty$, the group of isometries $\Gamma_\infty$, the isometry $g_0$, the distance function $d_E$, and the points $A_0$, $B_0$, $B_t$, $C_t$ and $D_t$.
They will be defined in this section.

We will explain that the ends of axis of $t\circ g$ are limited in some certain small open balls when $t$ is `large' enough, where $g$ is some certain isometry and $t$ changes in the set $\Gamma_\infty$ which will be defined later. We estimate the location of the axis of $t\circ g$.

Let $M=\mathbb{H}^n/ \Gamma$ be a nonelementary hyperbolic $n$-manifold with at least one cusp, where $\Gamma\cong \pi_1(M)$ is a nonelementary discrete torsion-free subgroup of the isometry group of hyperbolic $n$-space $\mathrm{Isom}(\mathbb{H}^n)$.
We choose a cusp $\mathcal{C}$ of $M$.
Since $\Gamma$ is a nonelementary, we can always take the maximal embedded horocusp $\mathcal{T}$ of $\mathcal{C}$ whose boundary contains a finite numbers of self-tangencies.
Pick a distinguished self-tangency $A$. 

Let $\widetilde{A}$ be a lift of $A$ under the covering projection $\pi: \widetilde{M}=\mathbb{H}^n\to M$. $\mathcal{H}$ and $\mathcal{H'}$ are two horoballs which are preimages $\pi^{-1}(\mathcal{T})$ of $\mathcal{T}$ that intersect at $\widetilde{A}$. Since $\mathcal{H}$ and $\mathcal{H'}$ both cover $\mathcal{T}$, there must be some isometries relating them, namely, there is some isometry $g\in \Gamma$ such that $g(\mathcal{H})=\mathcal{H'}$.

\begin{figure}[h]
\centering  
\includegraphics[width=0.9\textwidth]{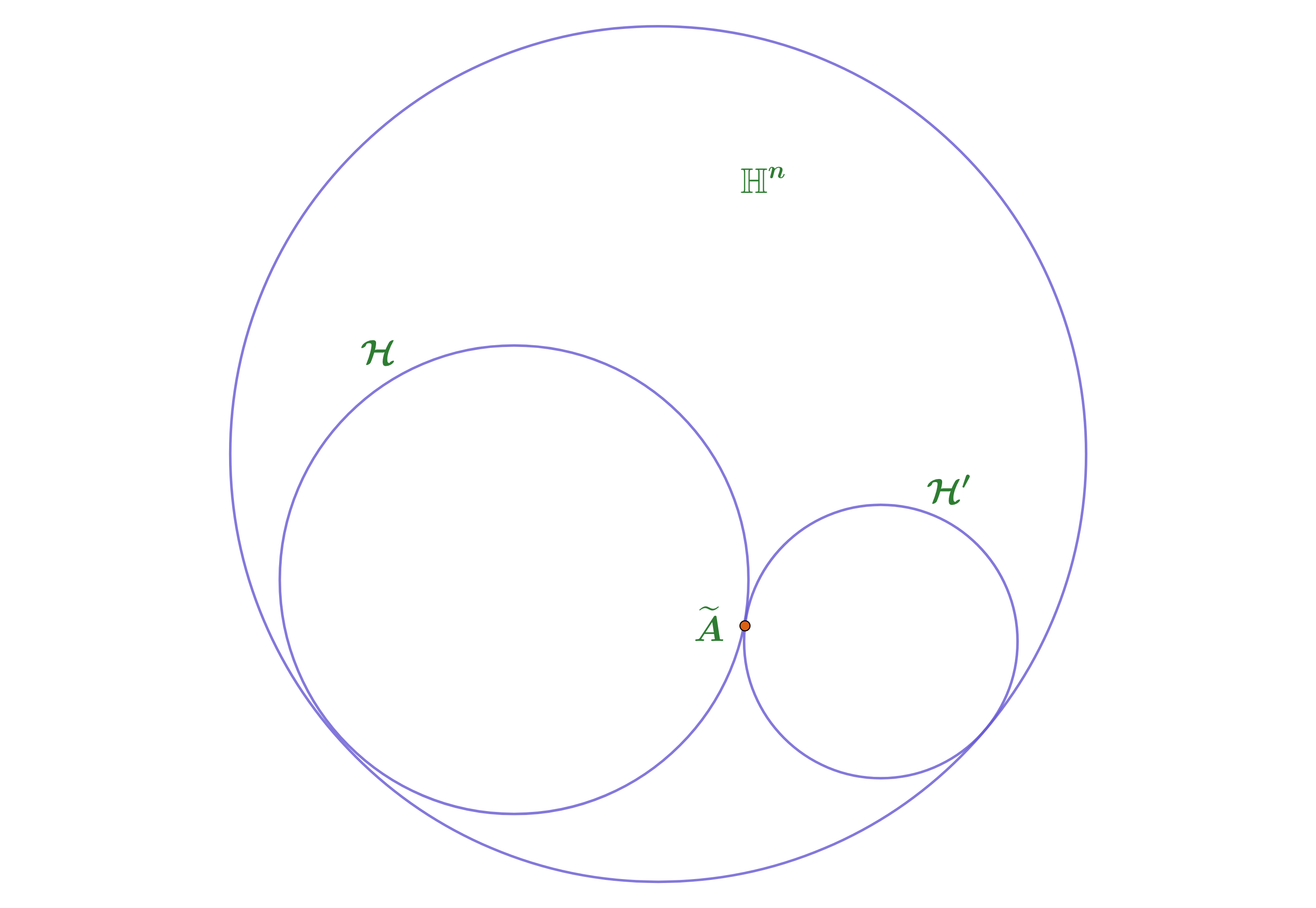}  
\caption{the maximal embedded horoballs}
\end{figure} 

For any isometries $g, g' \in \Gamma$ which map $\mathcal{H}$ to $\mathcal{H'}$, $g'\circ g^{-1}$ stabilizes $\mathcal{H'}$. It means that $g'=t\circ g$ and $t\in \Gamma$ is an isometry stabilizing $\mathcal{H'}$. We fix a distinguished $g$ and adjust $t$ to get a family of isometries $t\circ g$ which map $\mathcal{H}$ to $\mathcal{H'}$.

Consider the upper half-space model of $\mathbb{H}^n$, which is the space $\mathbb{H}^n=\{(x_1,\cdots,x_n)\in \mathbb{R}^n|x_n>0\}$ with the Riemannian metric $(d x_1^2+\cdots+d x_n^2) / x_n^2$. Without loss of generality, we assume that $\mathcal{H}$ is a Euclidean ball of radius $1$ centered at $(0,\cdots,0,\frac{1}{2})$, and $\mathcal{H'}$ is the domain $\{(x_1,\cdots,x_n)\in \mathbb{R}^n|x_n\ge 1\}$ above the hyperplane $x_n=1$. $\mathcal{H'}$ intersects $\mathcal{H}$ at the point $A_0=(0,\cdots,0,1)$. We rewrite $\mathcal{H}$ as $\mathcal{H}_0$, and $\mathcal{H'}$ as $\mathcal{H}_\infty$ below. 

Let $g_0$ be our fixed isometry that sends $\mathcal{H}_0$ to $\mathcal{H}_\infty$ and $\Gamma_\infty \subset \Gamma$ consist of all the isometries stabilizing $\mathcal{H}_\infty$.
We fix a fundamental domain $\mathcal{D} \subset \partial H_\infty$ contains $A_0$ with respect to the action of $\Gamma_\infty$, for example, a suitable parallelogram.
Possibly after modifying $g_0$ with a translation in $\Gamma_\infty$, we can assume that the point $B_0=g_0(A_0)$ lies in $\mathcal{D}$.
Denote $t\circ g_0$ as $g_t$ where $t\in \Gamma_\infty$ is an isometry stabilizing $\mathcal{H}_\infty$ and $B_t=g_t(A_0)=t(B_0)$. we use $\|t\|$ to denote $d_E(A_0,B_t)$, where $d_E$ stands for the Euclidean distance with respect to the induced Riemannian metric on $\partial \mathcal{H}_\infty$.

\begin{figure}[h]
\centering  
\includegraphics[width=1\textwidth]{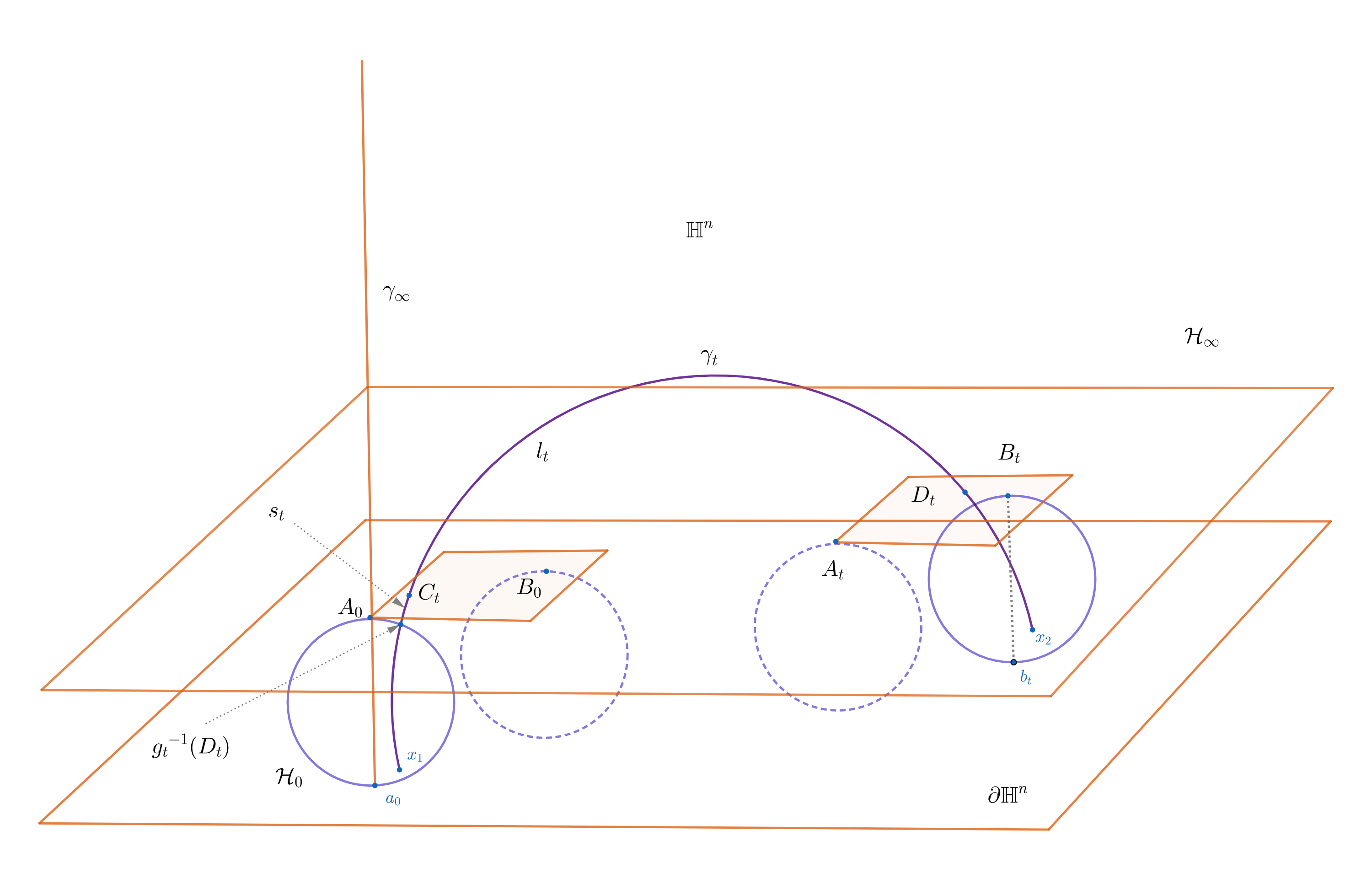}  
\caption{the upper half-space model}
\end{figure} 

Suppose $g_t$ is hyperbolic and denote the axis of $g_t$ as $\gamma_t$ which intersects $\partial \mathcal{H}_\infty=\{(x_1,\cdots,x_n)\in \mathbb{R}^n|x_n= 1\}$ at the points $C_t$ and $D_t$. Let $\gamma_\infty$ be the line which is perpendicular to $\partial {\mathbb{H}}^n$ through the point $A_0$. 
We will see that $C_t$, $D_t$ tend to $A_0$, $B_t$ respectively, then we can study the geodesic $\gamma_t$ via $A_0$ and $B_t$ when $\|t\|$ is large enough.

In \cite{MR2263061}, S. M. Kuhlmann use fractional linear transformations to get the similar result in $3$-dimensional case. 
For general dimension $n$, we switch to vector calculations in the upper half space model. 

\begin{lem}\label{lem1}
    For $t \in \Gamma_\infty$, $C_t\to A_0, D_t\to B_t$, when $\|t\|\to \infty$.
\end{lem}
    
    Here $D_t\to B_t$ means $d_E(D_t,B_t)$ will be arbitrarily small with $\|t\|$ increasing.
\begin{proof}
    We estimate the endpoints of $\gamma_t$.
    First we describe $g_t$ by some simple isometries.

    Denote the inversion about the Euclidean $(n-1)$-sphere of radius $1$ centered at the origin as $h_1$:
        $$h_1(\overrightarrow{X})=\frac{\overrightarrow{X}}{\|\overrightarrow{X}\|^2},$$
    where $\|\cdot\|$ means the Euclidean norm $\sqrt{(\overrightarrow{x},\overrightarrow{x})}$ in vector space as usual.

    Denote the reflection about a Euclidean $(n-1)$-plane orthogonal to $\partial \mathbb{H}^n$ which sends $A_0$ to $B_0$ as $h_2$:
    $$h_2(\overrightarrow{X})=\overrightarrow{X}-\frac{2(\overrightarrow{X},\overrightarrow{b})}{\|\overrightarrow{b}\|^2}\overrightarrow{b}+\overrightarrow{b},where \overrightarrow{b}=\overrightarrow{A_0B_0}.$$

    It is obvious that $h_1$ sends $\mathcal{H}_0$ to $\mathcal{H}_\infty$ fixing $A_0$, and $h_2$ keeps $\mathcal{H}_\infty$ sending $A_0$ to $B_0$.
    Let $h=h_2\circ h_1$.
    Then we see that $h(\mathcal{H}_0)=\mathcal{H}_\infty$ and $h(A_0)=B_0$.

    Notice that $g_0$ sends $\mathcal{H}_0$ to $\mathcal{H}_\infty$ and $g_0(A_0)=B_0$, so $h^{-1}g_0$ is an isometry keeping $A_0$ and $\mathcal{H}_\infty$. 
    Therefore, $h^{-1}g_0$ fixes all plane that parallel with $\partial \mathbb{H}^n$.
    So $h^{-1}g_0\in O(\mathbb{R}^n)$, where $O(\mathbb{R}^n)$ consists of orthogonal $n\times n$ matrices. We denote $h^{-1}g_0$ as $V$.

    Now we have $g_0=h\circ V, g_t=t\circ h\circ V$ where $t$ is a Euclidean isometry stabilizing $\mathcal{H}_\infty$. Assume $t(\overrightarrow{X})=A\overrightarrow{X}+\overrightarrow{c}$ where $A \in O(\mathbb{R}^n)$.

    To solve the endpoints of $\gamma_\infty$, let $g_t(\overrightarrow{X})=\overrightarrow{X}$. Then we obtain

    $$A(\frac{V\overrightarrow{X}}{\|V\overrightarrow{X}\|^2}-\frac{2(V\overrightarrow{X},\overrightarrow{b})}{\|V\overrightarrow{X}\|^2\|\overrightarrow{b}\|^2}\overrightarrow{b}+\overrightarrow{b})+\overrightarrow{c}=\overrightarrow{X}$$
    $$\overrightarrow{X}-A\frac{V\overrightarrow{X}}{\|V\overrightarrow{X}\|^2}+\frac{2(V\overrightarrow{X},\overrightarrow{b})}{\|V\overrightarrow{X}\|^2\|\overrightarrow{b}\|^2}A\overrightarrow{b}=A\overrightarrow{b}+\overrightarrow{c}=\overrightarrow{A_0B_t}$$

    Notice that $A,V\in O(\mathbb{R}^n)$, so: 
    \begin{equation}
        \begin{aligned}
        \|\overrightarrow{A_0B_t}\|
        &\le\|\overrightarrow{X}\|+\left\| \frac{AV\overrightarrow{X}}{\|V\overrightarrow{X}\|^2}\right\|+\left\|\frac{2(V\overrightarrow{X},\overrightarrow{b})}{\|V\overrightarrow{X}\|^2\|\overrightarrow{b}\|^2}V\overrightarrow{b}\right\|\\
        &\le\|\overrightarrow{X}\|+\frac{3}{\|\overrightarrow{X}\|}
        \end{aligned}
    \end{equation}

    and
    \begin{equation}
        \begin{aligned}
        \|\overrightarrow{A_0B_t}-\overrightarrow{X}\|
            &\le\left\|\frac{AV\overrightarrow{X}}{\|V\overrightarrow{X}\|^2}\right\|+\left\|\frac{2(V\overrightarrow{X},\overrightarrow{b})}{\|V\overrightarrow{X}\|^2\|\overrightarrow{b}\|^2}V\overrightarrow{b}\right\|\\
            &\le\frac{3}{\|\overrightarrow{X}\|}
        \end{aligned}
    \end{equation}

    When $\|t\|=\|\overrightarrow{A_0B_t}\|\to \infty$, we get $\|\overrightarrow{X}\|\to 0$ or $\infty$ by $(1)$ and $\|\overrightarrow{A_0B_t}-\overrightarrow{X}\|\to 0$ as $\|\overrightarrow{X}\|\to \infty$ by $(2)$.

    Denote the two roots as $x_1$ and $x_2$, then $x_1\to a_0$, $x_2\to b_t$ as $\|t\|\to \infty$, where $a_0$, $b_t$ are the images of $A_0$ ,$B_t$ respectively under the orthogonal projection from $\partial \mathcal{H}_\infty$ to $\partial \mathbb{H}^n$.

    Recall that $C_t, D_t$ are the intersection of $\gamma_t$ and $\partial \mathcal{H}_\infty$. We get $C_t\to A_0$ and $D_t\to B_t$.
\end{proof}

Consider the projection image $\pi(\gamma_t)$ of $\gamma_t$ in $M$.
We will show that it contains two distinct arcs. Denote the arc $g_t^{-1}(D_t)C_t$ as $s_t$ and arc $C_tD_t$ as $l_t$. Then $\pi(s_t)$ is the part of $\pi(\gamma_t)$ outside $\mathcal{T}$ and $\pi(l_t)$ is the part inside $\mathcal{T}$, where $\mathcal{T}$ is the maximal embedded horocusp and $\pi$ is the covering map of $\mathbb{H}^n\to M$ mentioned before.

Intuitively, when $\|t\|$ is sufficiently large, $\pi(s_t)$ will be a short geodesic segment very close to the self-tangency of the maximal horocusp, so $\pi(s_t)$ has no self-intersection.
The following Lemma \ref{lem2} gives a sufficient condition for $\pi(l_t)$ to have no self-intersection.
Recall that $\Gamma_\infty \subset \Gamma$ consists of all the isometries stabilizing $\mathcal{H}_\infty$.
Then $\Gamma_\infty$ stabilizes every Euclidean hyperplane parallel to $\partial H_\infty$.

\begin{lem}\label{lem2}
    When $\|t\|$ is large enough and if there is no points $M, N$ on the arc $l_t = C_tD_t$ such that there is an isometry $\tau\in \Gamma_\infty$ satisfying $\tau(M)=N$, then $\pi(\gamma_t)$ is simple closed geodesic.
    Moreover, the lemma is also true if we replace the arc $l_t = C_tD_t$ by the projected line segment $C_tD_t$ on $\partial\mathcal{H}_\infty$.
\end{lem}

\begin{proof}
    By Lemma \ref{lem1}, $C_t$ and $g_t^{-1}(D_t)$ approach $A_0$ as $\|t\|\to \infty$. $\pi(s_t)$ is embedding because $\Gamma$ is discrete.

    If $\pi(l_t)$ is not embedding, there must be some points $M, N$ such that $\tau(M)=N$ where $\tau \in \Gamma_\infty$. This is the first claim.
    Then we consider the case of the projected line segment $C_tD_t$. Denote $M', N'$ to be the images of $M$ and $N$ under projection. According to the given conditions, there are no maps from $M'$ to $N'$, hence there are no maps from $M$ to $N$ neither.
\end{proof}

\section{The $n$-dimensional case with toral cusps}\label{sec3}

We are now ready to prove the first theorem in the paper which is Theorem \ref{ThmZnCusp0} in Introduction:

\begin{thm}\label{thm1}
    A cusped nonelementary hyperbolic $n$-manifold contains infinitely many simple closed geodesics if there is a cusp which is a flat $(n-1)$-manifold $\mathbb{R}^{n-1}/\Gamma_c$ such that $\Gamma_c \cong \mathbb{Z}^{n-1}$ for $n\ge 3$.
\end{thm}   

In this case, $\Gamma_\infty$ corresponding to this cusp is generated by $(n-1)$ translations when we restrict isometries in $\Gamma_\infty$ to the boundary of the horoball $\mathcal{H}_\infty$.

\begin{proof}
    It suffices to show that there are infinitely many $\gamma_t$ whose $l_t$ is embedded by Lemma \ref{lem2}.

    Let $\Gamma_\infty$ is generated by $t_1, t_2, \cdots, t_{n-1}$ which are all linearly independent translations. Consider $\partial \mathcal{H}_\infty = \mathbb{R}^{n-1} $ and take the translation vectors of $t_1, t_2, \cdots, t_{n-1}$ as a basis in it.

    Let $A_0=(0, \cdots, 0), B_0=(b_1,\cdots,b_{n-1})$ where $0\le b_i < 1$ because we set $B_0$ in tha same fundamental domain of $\mathcal{H}_\infty$ with $A_0$.
    It is also easy to see that not all $b_i$ equal to $0$. We assume $b_2\neq 0$ without loss of generality.

    Take the isometry $\tau_n=t_1^n$, so $B_{\tau_n}=(b_1+n,b_2,\cdots,b_{n-1})$. 
    We assume the $l_{\tau_n}$ is not embedded to get a contradiction. Assume there are points $M, N$ in the line segment $C_{\tau_n}D_{\tau_n}$ on $\partial \mathcal{H}_\infty$ and there is a $t=\sum_{i=0}^{n-1}w_it_i\in \Gamma_\infty$ such that $t(M)=N$ where $w_i\in \mathbb{Z}$. Then $(N_1,\cdots,N_{n-1})=(M_1+w_1,\cdots,M_{n-1}+w_{n-1})$ where $M_i, N_i$ is the $i$-th coordinate of $M, N$.

    Due to Lemma \ref{lem1}, we take $n$ large enough such that $\|C_{\tau_n}-A_0\|, \|D_{\tau_n}-B_{\tau_n}\|<\varepsilon$ for some $\varepsilon$ small suffiently.
    Then $\|M_i-N_i\|\le \|(C_{\tau_n})_i-(D_{\tau_n})_i\|\le b_i+2\varepsilon<1$, we know that $w_i=0$, $2\le i\le n-1$.

    Then $t=w_1t_1$. 
    Hence $\overrightarrow{MN} \parallel t_1$, here we denote the translation vector of $t_1$ as $\overrightarrow{t_1}$.
    Since $M$ and $N$ are on the line segment $C_{\tau_n}D_{\tau_n}$, $C_{\tau_n}D_{\tau_n}\parallel \overrightarrow{t_1}$, too.
    However, $\|(C_{\tau_n})_2-(D_{\tau_n})_2\|\ge b_2-2\varepsilon>0$ since $b_2\neq 0$.
    It means that $\overrightarrow{MN} \nparallel \overrightarrow{t_1}$. This is a contraction.

    So $\gamma_{\tau_n}$ is a simple closed geodesic when $n$ is large enough. This leads to the conclusion. 
\end{proof}

\begin{remark}
    The method follows \cite{MR2263061}. In \cite{MR2263061}, S.M. Kuhlmann first use fractional linear transformation to get the result of Lemma \ref{lem1} in $3$-dimensional case. Then she proposes a certain family of translations $B_{\tau_n}$ and use this method to finish her proof, since the cusp of cusped orientable hyperbolic $3$-manifold of finite-volume is a torus, which means $\Gamma_c\cong \mathbb{Z}^2$. 
\end{remark}

The main theorem of \cite{MR2263061} is a special case of Theorem \ref{thm1}:

\begin{coro}\label{coro1}
    Any cusped finite-volume orientable hyperbolic $3$-manifold contains infinitely many simple closed geodesics.
\end{coro}

\section{The non-orientable $3$-dimensional case}\label{sec4}

In this section, we consider the case of nonelementary hyperbolic $3$-manifolds with a virtually rank-$2$ cusp. First, we discuss all the elements in $\Gamma_\infty$ if there are some non-orientable transformations in it.

If there is an isometry $s\in \Gamma_\infty$ and $s$ is orientation-reversing, then $s$ is a glide reflection. Since $\Gamma_\infty$ act freely, all the translation vectors of glide reflections in $\Gamma_\infty$ are parallel to the translation vector of $s$.
Suppose $s$ contains the shortest translation distance without loss of generality. 
Take an appropriate coordinate system, we can set $s(\vec{X})=\begin{pmatrix}1&0\\0&-1\end{pmatrix}\vec{X}+\begin{pmatrix}\alpha\\0\end{pmatrix}$, where $\alpha>0$.

Since the hyperbolic manifolds we consider have a virtually rank-$2$ cusp, there exists some translation whose vector is not parallel to the vector of $s$.
For any translation $t'(\vec{X})=\vec{X}+\begin{pmatrix}a\\b\end{pmatrix}\in\Gamma_\infty,(b\neq0)$, $st's^{-1}(\vec{X})=\vec{X}+\begin{pmatrix}a\\-b\end{pmatrix}$ and $\vec{X}+\begin{pmatrix}0\\2b\end{pmatrix}=(t'-sts^{-1})(\vec{X})\in \Gamma_\infty$.
Then the translation vector of $t'-sts^{-1}$ is perpendicular to $\begin{pmatrix}\alpha\\0\end{pmatrix}$.

Take a translation of smallest distance among those translations whose vector is perpendicular to $\begin{pmatrix}\alpha\\0\end{pmatrix}$. Denote it as $t(\vec{X})=\vec{X}+\begin{pmatrix}0\\\beta\end{pmatrix}$.
Then $s$, $t$ are the generators of $\Gamma_\infty$.
For any $w \in\Gamma_\infty$, we can write it as $w=s^nt^m$ since $tst=s$. Then we can write down all the elements in $\Gamma_\infty$:

If $n=2k$, $w(\vec{X})=\vec{X}+\begin{pmatrix}n\alpha\\m\beta\end{pmatrix}$ is a translation.

If $n=2k+1$, $w(\vec{X})=\begin{pmatrix}1&0\\0&-1\end{pmatrix}\vec{X}+\begin{pmatrix}n\alpha\\-m\beta\end{pmatrix}$ is a glide reflection whose reflection axis is $y=-m\beta/2$ and translation distance is $n\alpha$ along $x$-axis. So all the reflection axis of glide reflection in $\Gamma_\infty$ is $m\beta/2$ for some $m\in\mathbb{Z}$. 

With the above description of a nonorientable cusp in mind, we prove the following theorem.

\begin{thm}\label{thm2}
    Any nonelementary hyperbolic $3$-manifold with a virtually rank-$2$ cusp contains infinitely many simple closed geodesics.
    In particular, any cusped finite-volume hyperbolic $3$-manifold contains infinitely many simple closed geodesics.
\end{thm}

\begin{proof}
    It is suffices to show the case if there are some orientation-reversing transformations in $\Gamma_\infty$.

    Due to the discussion above, we assume that $\Gamma_\infty$ is generated by the glide reflection $s$ and the translation $t$.
    Since we have showed that the translation vector of $s$ and $t$ are perpendicular, consider the fundamental region $\mathcal{D}=\{(x,y)~|~0\le x<\alpha,-\beta/2\le y<\beta/2\}$ in $\partial \mathcal{H}_\infty=\mathbb{R}^2$ with respect to the action of $\Gamma_\infty$. We have points $A_0,B_0\in \mathcal{D}$ and $A_0\neq B_0$. Below we consider different cases according to the position of $A_0$ and $B_0$.

    $\textbf{Case 1}$. Suppose that the x-coordinates of $A_0$ and $B_0$ are different.
    Consider $\tau_n(\vec{X})=\vec{X}+\begin{pmatrix}0\\n\beta\end{pmatrix}$ for any integer $n$. 
    Then $\pi(\gamma_{\tau_n})$ is a simple closed geodesic when $n$ is large enough as we proved in Theorem \ref{thm1}.

    \begin{figure}[h]
    \centering  
    \includegraphics[width=0.75\textwidth]{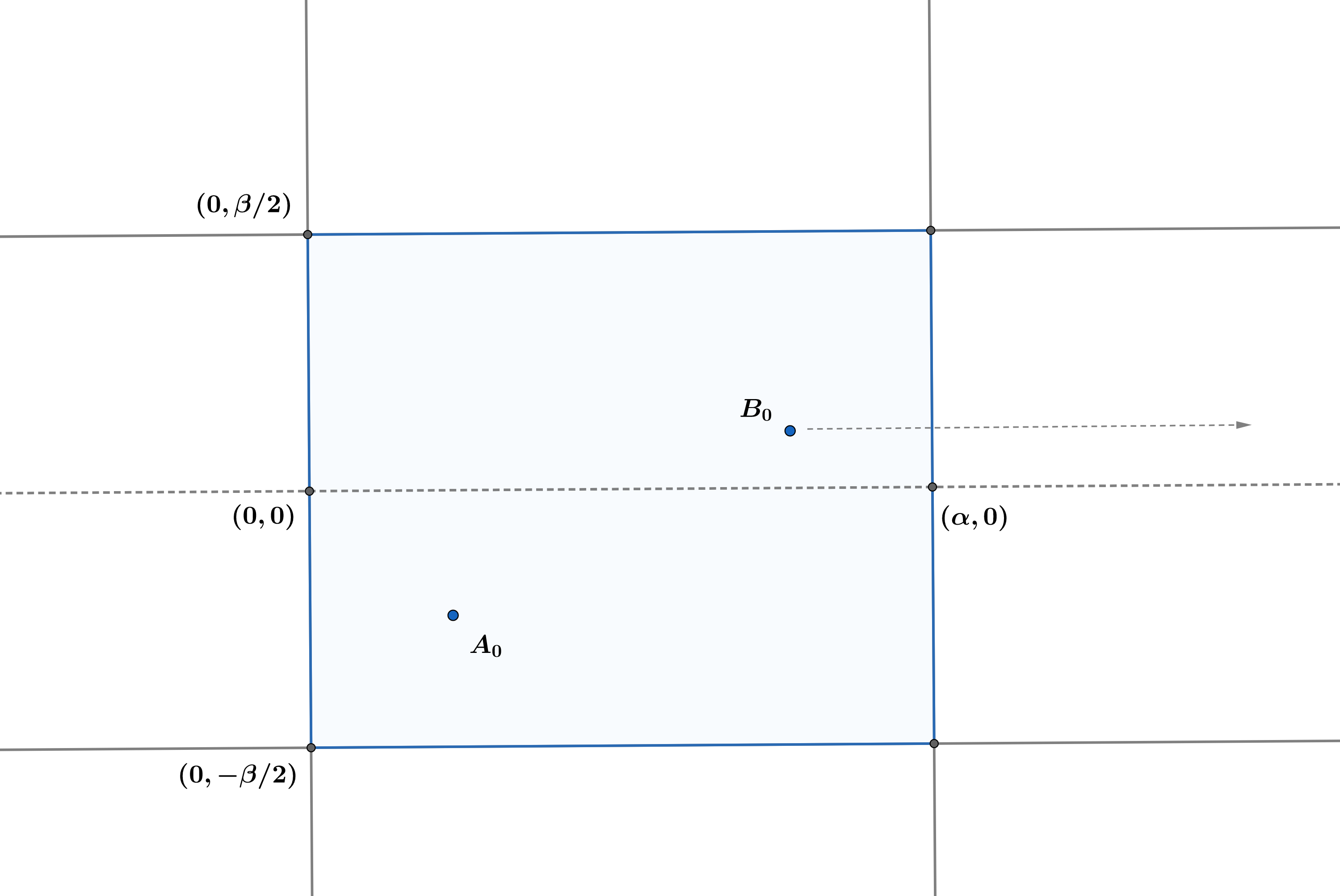}  
    \caption{Case 1}
    \end{figure} 

    $\textbf{Case 2}$. Suppose that the x-coordinates of $A_0$ and $B_0$ are the same. Hence their y-coordinates are different.
    Without loss of generality, we can assume that $A_0$ and $B_0$ are both in the y-axis.

    $\textbf{Case 2.1}$. We consider the case that their y-coordinates are symmetric about $y=0$ first. 

    Take $\tau_n=t^n\cdot s$ and $n$ large enough such that $d_E(C_{\tau_n},A_0), d_E(D_{\tau_n},B_{\tau_n})<\varepsilon$ for $\varepsilon$ small sufficiently where $d_E$ is the Euclidean distance on $\partial \mathcal{H}_\infty$.

    For all $M, N$ on the line segment $C_{\tau_n}D_{\tau_n}$, 
    \begin{equation}
        \begin{aligned}
            0<\|(M)_1-(N)_1\|
            & \le\|(C_{\tau_n})_1-(D_{\tau_n})_1\|\\
            & \le\|(A_0)_1-(B_{\tau_n})_1\|+2\varepsilon\\
            & =\alpha+2\varepsilon<2\alpha\nonumber
        \end{aligned}
    \end{equation}

    So if there is a $w\in \Gamma_\infty$ such that $w(M)=N$, $w$ will not be a translation since all the distance of translations along x-axis in $\Gamma_\infty$ are even number multiples of $\alpha$. 
    Then we assume $w$ is a glide reflection whose axis is $m\beta/2$ for some $m\in \mathbb{Z}$

    \begin{figure}[h]
    \centering  
    \includegraphics[width=0.75\textwidth]{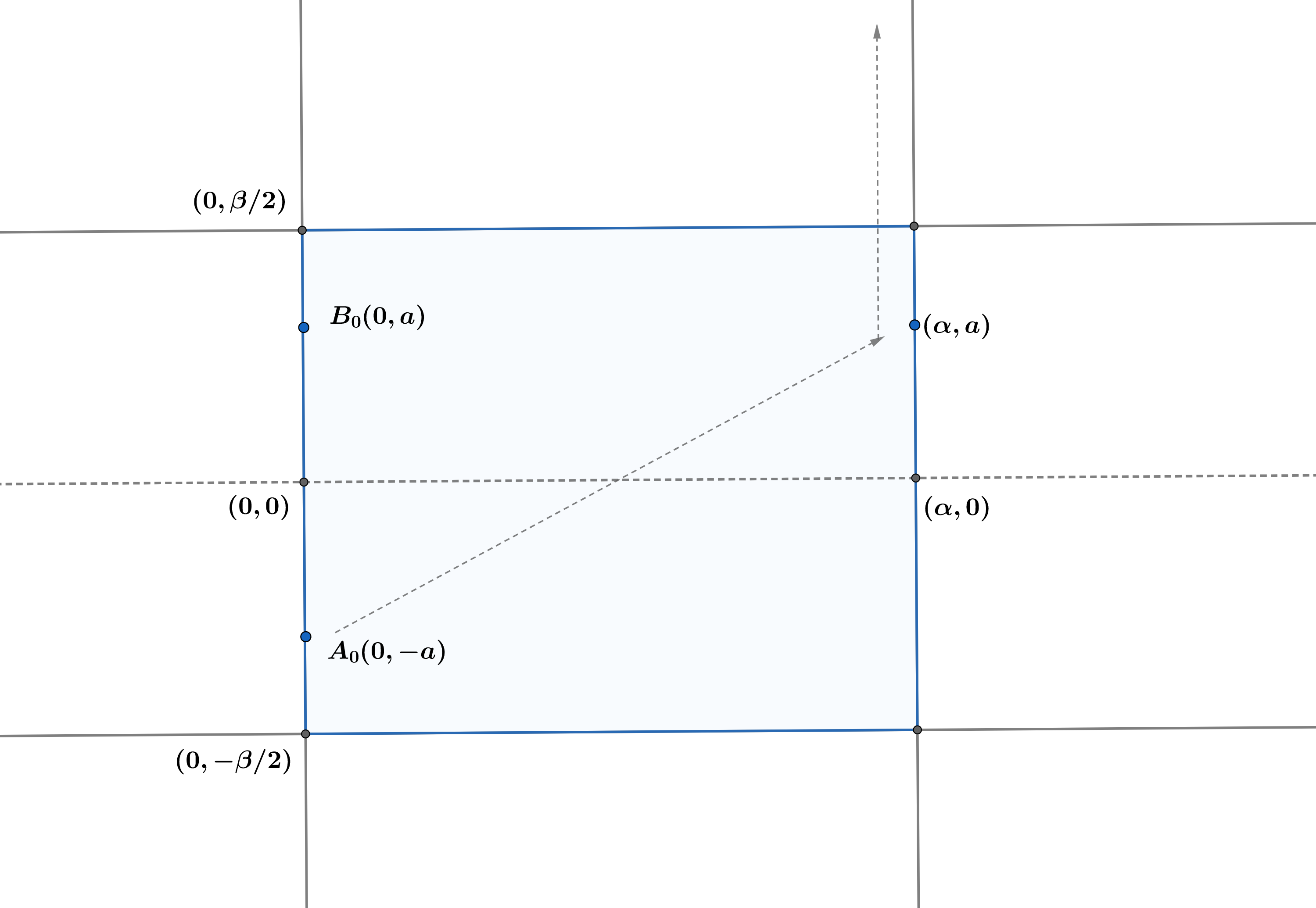}  
    \caption{Case 2.1}
    \end{figure} 

    Let $A_0=(0,-a), B_0=(0,a), 0<a<\frac{\beta}{2}$, then $B_{\tau_n}=(\alpha,-a+n\beta)$.  
    The y-coordinate of the middle point of $A_0$ and $B_{\tau_n}$ is $(-a+\frac{n\beta}{2})$. 
    Then we should discuss the points on the geodesic segment $l_t$ whose endpoints are $C_0$ and $D_{\tau_n}$ in place of the line segment $C_0D_{\tau_n}$. We will proof that if we take $n$ such that $0<|-a\pm\varepsilon|< \frac{\beta}{2}$, $l_t$ will not have self-intersections under any transformations in $\Gamma_\infty$.

    We consider the points $\bar{M}$ and $\bar{N}$ on the geodesic segment $l_t$ 
    The images of projection on $\partial \mathcal{H}_\infty$ of them are $M$ and $N$.
    If there is some isometry $w\in \Gamma_\infty$ such that $w(M)=N$, $\bar{M}$ and $\bar{N}$ must be on the same plane parallel with $\partial \mathcal{H}_\infty$.
    It means they are equidistant from the middle point $\bar{P}$ of $C_0$ and $D_{\tau_n}$ on the geodesic segment $l_t$. 

    \begin{figure}[h]
    \centering  
    \includegraphics[width=1\textwidth]{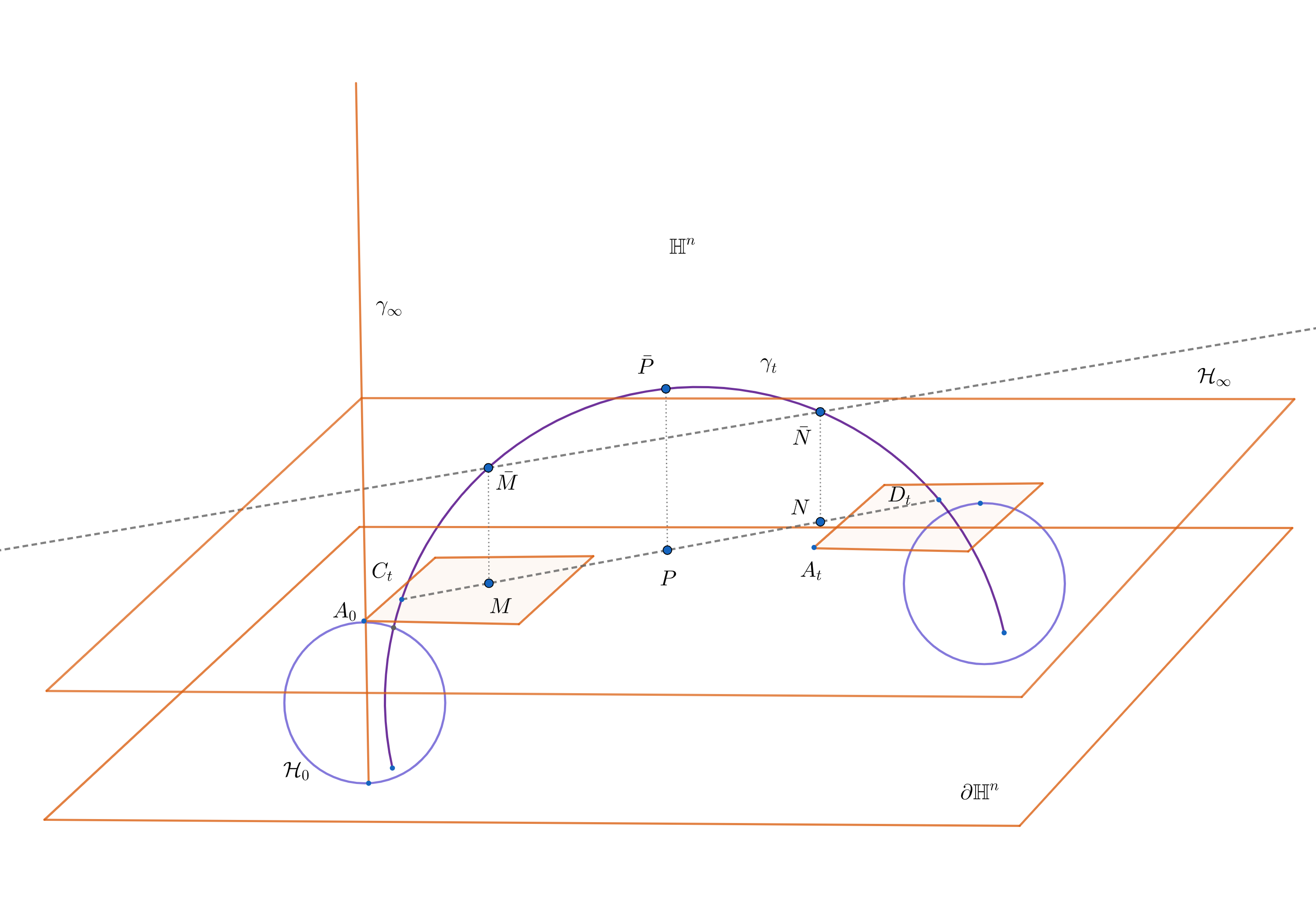}  
    \caption{$M$ and $N$ are symmetric}
    \end{figure} 

    Then the projected image $P$ of $\bar{P}$ on the $\partial\mathcal{H}_\infty$ is also the middle point of $C_0$ and $D_{\tau_n}$ on $\partial\mathcal{H}_\infty$, then $-a+\frac{n\beta}{2}-\varepsilon\le (P)_2\le -a+\frac{n\beta}{2}+\varepsilon$ because $d_E(C_{\tau_n},A_0), d_E(D_{\tau_n},B_{\tau_n})<\varepsilon$. However, $w$ is a glide reflection whose axis must be $m\beta/2$ for some $m\in \mathbb{Z}$. So $(P)_2=m\beta/2$.

    Now we take $n$ large enough such that $0<|-a\pm\varepsilon|< \frac{\beta}{2}$ to construct a contradiction.
    Hence $\pi(\gamma_{\tau_n})$ is a simple closed geodesic by Lemma \ref{lem2}.

    $\textbf{Case 2.2}$. Now we consider the situation that $A_0$ and $B_0$ are not symmetric about $y=0$. 

    \begin{figure}[h]  
    \centering  
    \includegraphics[width=0.75\textwidth]{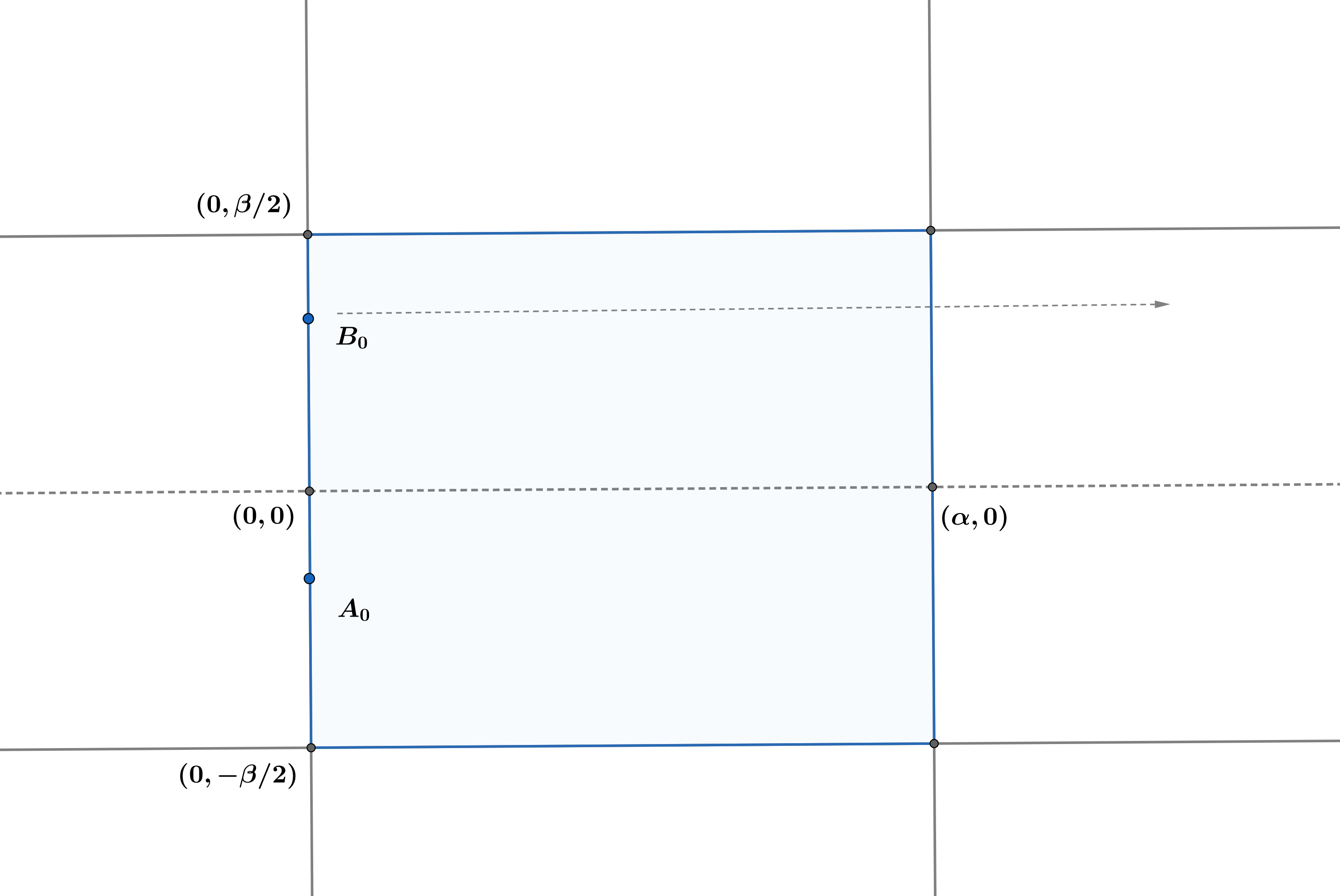}  
    \caption{Case 2.2}
    \end{figure} 

    We eliminate glide reflections first.
    According to the prove above, we only need to show that there are no reflection axis such that the $y$-coordinate of $A_0$ and $B_t$ are symmetric about it.
    Take $\tau_{2n}(\vec{X})=s^{2n}(\vec{X})=\vec{X}+\begin{pmatrix}2n\alpha\\0\end{pmatrix}$.
    Then the $y$-coordinate of $B_{\tau_{2n}}=\tau_{2n}(B_0)$ and $A_0$ are not symmetric about any reflection axis since we have assumed that $A_0$ and $B_0$ are not symmetric about $y=0$. So there is no glide reflection $w$ such that $l_t$ has self-intersections.

    On the other hand, the y-coordinate of $A_0$ and $B_{\tau_{2n}}$ are different, as well as $C_{\tau_{2n}}$ and $D_{\tau_{2n}}$ when we take $n$  large enough, so there is no translation $w'$ and points $M, N$ on the line segment $C_{\tau_{2n}}D_{\tau_{2n}}$ such that $w'(M)=N$ as the argment in Theorem \ref{thm1}.

    With Corollary \ref{coro1} and discussions in Case 1, Case 2.1 and Case 2.2 above, we complete the proof of Theorem \ref{thm2}.
\end{proof}

\section{Conclusions}\label{sec5}

The method is based on the picture of the structure of cusp.
It is known that the cusp of a hyperbolic $n$-manifold admits a flat metric.
It means that when we consider $n$-dimension cusped hyperbolic case, we should know the classification of $(n-1)$-dimension flat manifolds.
It seems that the method of this paper will likely work in $4$-dimensional case due to our knowledge about $3$-dimension flat manifolds.
We will continue to consider if any cusped finite-volume hyperbolic $4$-manifolds contain infinitely many simple closed geodesics.

However, the complexity of the structure of flat manifolds increases rapidly as the number of dimensions increases.
So when we consider the location of self-tangencies to construct simple closed geodesics, it will appear a lot of complications.
It seems that there will always be some special cases such that we have very few options to construct simple closed geodesics.
This will cause a lot of difficulties for our case-by-case approach to extend our conclusions to higher dimensions directly.

\bibliographystyle{plain}
\bibliography{wenxian}

\end{document}